\documentclass[reqno]{amsart}

\begin{document}

\vspace*{-0.5in}
\begin{center}\Small{Mathematical Inequalities \& Applications,  \textbf{10}  (2007), pp. 723--726\\
http://dx.doi.org/10.7153/mia-10-66} \end{center}
 \vspace*{0.5in}

\numberwithin{equation}{section}
\newtheorem{theorem}{Theorem}[section]

\newtheorem{lemma}[theorem]{Lemma}

\theoremstyle{remark}
\newtheorem*{remark}{Remark}

\setcounter{page}{723}

\title[On bounds of matrix eigenvalues]{On bounds of matrix eigenvalues}

\author{Jinhai Chen}
\address{Hong Kong Polytechnic University, Kowloon, Hong Kong}

\email{majhchen@gmail.com}

\subjclass[2000]{15A18, 15A60.}

\keywords{Trace, Rank, Eigenvalue}

\begin{abstract}
In this paper, we give estimates for both upper and lower bounds of eigenvalues  of a simple
matrix. The estimates are shaper than the known results.
\end{abstract}

\maketitle

\section{Introduction}\label{Int}
As is well known, the eigenvalues of a matrix play an important role
in solving linear systems \cite{OA94, AG97, JW01}, especially in the
perturbation problems \cite{GHCF94, RSV62}. The purpose of this note
is to give a specific estimate of the eigenvalues.

Let $A=(a_{ij})$ be an $n\times n$ complex matrix with conjugate
transpose $A^*$, $\overline{A}$ denote the conjugate, and ${\rm tr}
A$ represent the trace of matrix $A$. Let $\lambda_1,
\lambda_2,\cdots, \lambda_{n}$ be the eigenvalues of $A$, then
$$\sum_{i=1}^{n} |\lambda_i|^2 \leq \|A\|^2=\sum_{i,j=1}^{n}|a_{ij}|^2= {\rm tr}(AA^*),$$
where $\|A\|$ denotes the Frobenius norm of $A$. Let

$$\Re_A=\frac{A+A^*}{2},$$
$$\Im_A=\frac{A-A^*}{2i},$$
we call $\Re_A$ the Hermitian real part and $\Im_A$ the Hermitian
imaginary part of $A$. Let
$$q_A=\|A\|^2 -\frac{|{\rm tr} (A)|^2}{n},$$
$$\Delta_A=\frac{\|AA^*-A^*A\|^2}{2}.$$

\section{Main theorem}

\begin{theorem}\label{t10}
Suppose $\lambda$ is an eigenvalue of an $n\times n$ complex matrix
$A$ with geometric multiplicity $t$, then\begin{equation}\label{00}
\left|\lambda-\frac{{\rm tr}(A)}{n}\right|\leq
\sqrt{\frac{n-t}{(2n-t)t}}\sqrt{\frac{n-t}{n}q_A+\sqrt{q_A^2-\frac{(2n-t)t}{n^2}\Delta_A}}.
\end{equation}
 \end{theorem}

\begin{theorem}\label{c11}
Suppose $\lambda_{\Re_A}$, $\lambda_{\Im_A}$ are the eigenvalues of
an $n\times n$ complex matrices $\Re_A$ and $\Im_A$ with geometric
multiplicity $t$, respectively,  then
\begin{equation}\label{10}
\left|\lambda_{\Re_A}-\frac{{\rm tr} (\Re_A)}{n}\right|\leq
\sqrt{\frac{n-t}{nt}q_{\Re_A}},
\end{equation}
\begin{equation}\label{11}
\left|\lambda_{\Im_A}-\frac{{\rm tr} (\Im_A)}{n}\right|\leq
\sqrt{\frac{n-t}{nt}q_{\Im_A}}.
\end{equation}
 \end{theorem}

\section{Proof of theorem}

Before giving the proof of Theorems \ref{t10} and \ref{c11}, we
present some lemmas.

\begin{lemma}[see \cite{RHR74}]\label{l10}
Let $\lambda_1, \lambda_2,\dots, \lambda_{n}$ be the eigenvalues of
an $n\times n$ complex matrix $A$, then
$$\sum_{j=1}^{n}|\lambda_j|^2\leq\sqrt{\|A\|^4-\Delta_A}.$$
 \end{lemma}

\begin{lemma}\label{l11}
Let $A$ be an $n\times n$ complex matrix, ${\rm rank} (A)$ represent the rank of
$A$. Then
$$|{\rm tr} (A)|^2 \leq {\rm rank} (A)\sqrt{\|A\|^4-\Delta_A}.$$
 \end{lemma}

\begin{proof}
Let $\lambda_1, \lambda_2,\cdots, \lambda_{n}$ be the eigenvalues of
$A$, $\Lambda={\rm diag} (\lambda_1, \lambda_2,\cdots,
\lambda_{n})$. Suppose that the number of nonzero eigenvalues is $k$. Without loss of  generality, we can denote the nonzero eigenvalues of $A$ by $\lambda_1,
\lambda_2,\cdots, \lambda_{k}$. Then it is easily seen that
$$k\leq {\rm rank} (A).$$

Now suppose that $R=\Lambda+M$ is a Schur triangular form of $A$, i.e.,
$A=U^*RU$, $U$ is unitary orthogonal, $\Lambda$ is diagonal and $M$
is upper triangular. From Lemma \ref{l10}, we have
$$\sum_{j=1}^{k}|\lambda_j|^2\leq\sqrt{\|A\|^4-\Delta_A}.$$
Then
$$|{\rm tr} (A)|^2 =\left|\sum_{j=1}^{k}\lambda_j\right|^2\leq k\sum_{j=1}^{k}|\lambda_j|^2 \leq {\rm rank} (A)\sum_{j=1}^{k}|\lambda_j|^2 \leq {\rm rank}
(A)
\sqrt{\|A\|^4-\Delta_A}.$$ This shows the validity of conclusion.
 \end{proof}

Next, we provide the proof of Theorem \ref{t10}.

\begin{proof}[Proof of Theorem {\rm\ref{t10}}]
Let $M=\lambda I-A$, where $I$ is the $n\times n$  identity  matrix,
$\lambda$ is the $t$ multiple eigenvalue of $A$.  Then we have
$${\rm rank} (M)={\rm rank} (\lambda I-A)\leq n-t,$$
and the following equality
$$\Delta_M=\frac{\|(\lambda I-A)({
\overline{\lambda}I-A^*})-({ \overline{\lambda}I-A^*})(\lambda
I-A)\|}{2}=\Delta_A.$$

From  Lemma \ref{l11}, we have
\begin{align}\label{06}|{\rm tr} (M)|^2  &\leq {\rm rank} (M)
\sqrt{\|M\|^4-\Delta_M}\nonumber\\ &\leq
(n-t)\sqrt{\|M\|^4-\Delta_M}\leq (n-t)\sqrt{\|\lambda
I-A\|^4-\Delta_A}.\end{align}
In addition, by simple manipulations, we obtain
\begin{align}\label{01}
|{\rm tr} (\lambda I-A)|^2&={\rm tr}(\lambda I-A){\rm tr}\left({
\overline{\lambda}I-A^*}\right)\nonumber\\
 &=n^2|\lambda|^2-n\lambda
{\rm tr}({A}^*)-n\overline{\lambda} {\rm tr} (A) +|{\rm
tr}(A)|^2=n\sigma+|{\rm tr} A|^2,
\end{align}
where $\sigma=n|\lambda|^2-\lambda {\rm
tr}({A}^*)-\overline{\lambda} {\rm tr} (A)$. Moreover,
\begin{align}\label{02}
\|\lambda I-A\|^4&=\left({\rm tr}\left(\left(\lambda I-A\right)\left({\lambda
I-A}\right)^*\right)\right)^2
\nonumber\\
&=\left(n|\lambda|^2-\lambda {\rm tr}({A}^*)-\overline{\lambda} {\rm tr}
(A) +\|A\|^2\right)^2=(\sigma+\|A\|^2)^2.
\end{align}
 Eliminating  $\sigma$ from the formulae (\ref{01}) and
(\ref{02}), we get
\begin{equation}\label{03}
\|\lambda I-A\|^4=\left(\frac{|{\rm tr}(\lambda I-A)|^2- |{\rm tr}
(A)|^2}{n}+\|A\|^2\right)^2.
\end{equation}
Let $s=\left|\lambda-\frac{{\rm tr} (A)}{n}\right|^2$,
$q_A=\|A\|^2-\frac{|{\rm tr} (A)|^2}{n}$. Then
\begin{equation}\label{add02}
|{\rm tr}(\lambda I-A)|^2=n^2s,\end{equation} and \begin{equation}\label{add03}
\|\lambda I-A\|^4=(ns+q_A)^2.\end{equation}
By substituting the equalities \eqref{add02} and \eqref{add03} into (\ref{06}), it follows that
$$n^2s\leq (n-t)\sqrt{(ns+q_A)^2-\Delta_A}.$$
Consequently, by straightforward computations, we have
$$s=\left|\lambda-\frac{{\rm tr} (A)}{n}\right|^2\leq {\frac{n-t}{(2n-t)t}}\left(\frac{n-t}{n}q_A+\sqrt{q_A^2-\frac{(2n-t)t}{n^2}\Delta_A}\right).$$
 The result follows immediately.
 \end{proof}

\begin{proof}[Proof of Theorem {\rm\ref{c11}}]
Notice that
$$\sqrt{\frac{n-t}{(2n-t)t}}\sqrt{\frac{n-t}{n}q_A+\sqrt{q_A^2-\frac{(2n-t)t}{n^2}\Delta_A}}\leq
\sqrt{\frac{n-t}{nt}q_{A}}.$$
Furthermore, the above equality holds  if and only if
$\Delta_A =0$. In other words,  $A$ is normal, i.e., $AA^*=A^*A$. By Theorem \ref{t10}, and taking into account that $\Re_A$ and $\Im_A$ are both normal matrices,  we get the validity of Theorem \ref{c11}.
\end{proof}

\begin{remark}\label{r10}
In terms of estimates on bounds of the largest modulus eigenvalue
$|\lambda|_{\max}$ of matrix $A$,  the
following inequality was given in \cite{HGP801, HGP802},
\begin{equation}\label{05}
\frac{|{\rm tr} (A)|}{n}\leq\left|\lambda\right|_{\max}\leq
\frac{|{\rm tr} (A)|}{n}+ \sqrt{\frac{n-1}{n}q_{A}}.
\end{equation}
We note that the estimates (\ref{00}), (\ref{10}), (\ref{11}) are
sharper than (\ref{05}) in some extent. That is to say, the results
presented in this paper improve the ones given in
\cite{HGP801, HGP802} partially,  and can be taken as
supplements to the conclusions known in \cite{JW01, HGP801, HGP802},
especially for the upper bound estimation of eigenvalues of a matrix.
\end{remark}

\end{document}